\documentclass[12pt]{amsart}
\usepackage[utf8]{inputenc}
\usepackage{amsmath,amssymb,amsthm,graphics}
\usepackage{colonequals}
\usepackage{bm}
\usepackage{diagmac2}
\usepackage{enumitem}
\usepackage{textcomp}
\usepackage[all,cmtip]{xy}
\usepackage{yfonts}
\usepackage{stmaryrd}
\usepackage{mathrsfs}
\usepackage{caption}
\usepackage[colorlinks,anchorcolor=blue,citecolor=blue,linkcolor=blue,urlcolor =blue,bookmarksdepth=2]{hyperref}
\usepackage{tikz}
\usetikzlibrary{cd,arrows}
\tikzset{>=stealth}
\tikzcdset{arrow style=tikz}
\tikzset{link/.style={column sep=1.8cm,row sep=0.16cm}}

\AtBeginDocument{%
	\def\MR#1{}
}
\usepackage{fancyhdr}
\fancypagestyle{titlepage}
{
	\fancyhf{}

	\fancyfoot[l]{
	\footnotesize
	\href{https://mathscinet.ams.org/mathscinet/msc/msc2010.html}
		{\emph{2010 Mathematics Subject Classification}}
		14N05, 14J70, 14G15.
	\smallskip
	\ContactInfo
	}
}

\newcommand{\bZ}{\mathbb{Z}}

\newcommand{\bN}{\mathbb{N}}

\newcommand{\bP}{\mathbb{P}}

\newcommand{\bF}{\mathbb{F}}

\newtheorem{theorem}{Theorem}[section]
\newtheorem{Teo}[theorem]{Theorem}
\newtheorem{Prop}[theorem]{Proposition}
\newtheorem{Lemma}[theorem]{Lemma}

\numberwithin{equation}{section}

\theoremstyle{definition}

\newtheorem{Def}[theorem]{Definition}
\newtheorem{Eg}[theorem]{Example}

\newtheorem{Rmk}[theorem]{Remark}

\begin{document}

\title{Transverse Lines to surfaces \\ over finite fields}
\author{Shamil Asgarli
	\and Lian Duan$^*$ 
	\and Kuan-Wen Lai}

\newcommand{\ContactInfo}{{
\tiny

\noindent S.~Asgarli,
\textsc{University of British Columbia,
Vancouver, V6T1Z2, Canada}\par\nopagebreak
\noindent\texttt{sasgarli@math.ubc.ca}
\noindent\textsc{ORCID: 0000-0003-3165-8297}

\noindent L.~Duan, 
\textsc{colorado state university,
Fort Collins, CO 80523, USA}\par\nopagebreak
\noindent\texttt{lian.duan@colostate.edu}
\noindent\textsc{ORCID: 0000-0003-4922-7831}

\noindent K.-W.~Lai,
\textsc{University of Massachusetts,
Amherst, MA 01003, USA}\par\nopagebreak
\noindent\texttt{lai@math.umass.edu}
}}

\begin{abstract}
We prove that if $S$ is a smooth reflexive surface in $\bP^3$ defined over a finite field $\bF_q$, then there exists an $\bF_q$-line meeting $S$ transversely provided that $q\geq c\deg(S)$, where $c=\frac{3+\sqrt{17}}{4}\approx 1.7808$. Without the reflexivity hypothesis, we prove the existence of a transverse $\bF_q$-line for $q\geq \deg(S)^2$.
\end{abstract}

\maketitle
\thispagestyle{titlepage}

\section*{Introduction}\label{sect:intro}

Given a smooth variety $X\subseteq\bP^{n}$ defined over an algebraically closed field $k$, a classical theorem of Bertini asserts that $X\cap H$ is smooth for a general hyperplane $H$ defined over $k$ \cite[Theorem~II.8.18]{Hartshorne}. The same proof in fact works for any infinite field $k$. When $k=\bF_q$ is a finite field, it is possible that $H\cap X$ is singular for every hyperplane $H$ defined over $\bF_q$. The following example is due to Nick Katz \cite{Katz-space-filling}:
\[
    S: X^q Y - X Y^q + Z^q W - Z W^q = 0
\]
defines a smooth surface in $\bP^3$ over $\bF_{q}$ such that each $\bF_q$-hyperplane is tangent to $S$; in particular, its hyperplane sections over $\bF_q$ are all singular (Example~\ref{katz-Example}).
While we cannot guarantee the existence of a smooth hyperplane section, Poonen \cite[Theorem 1.1]{Poonen-bertini} proved that there are plenty of hypersurfaces $Y\subseteq \bP^{n}$ such that $X\cap Y$ is smooth.

Another approach to remedy the original Bertini theorem in the case of finite fields is to investigate how large $q$ should be with respect to the invariants of the variety $X$ so that $X$ admits a favourable linear section. For instance, the first author \cite{Asgarli} proved that if $C\subseteq\bP^2$ is a smooth reflexive plane curve of degree $d$ over $\bF_q$ such that $q\geq d-1$, then there is an $\bF_q$-line which meets $C$ transversely. In this paper, we prove an analogous transversality result for surfaces.

\begin{Teo}\label{main-theorem}
Let $S\subseteq\bP^{3}$ be a smooth reflexive surface of degree $d$ defined over $\bF_q$. There exists an $\bF_q$-line $L\subseteq\bP^3$ meeting $S$ transversely provided that $q\geq cd$, where $c=\frac{3+\sqrt{17}}{4}\approx 1.780776$.
\end{Teo}

Recall that a line $L$ meets a surface $S$ \emph{transversely} if the intersection $S\cap L$ consists of $d=\deg(S)$ distinct geometric points. The reflexivity of a surface $S$ is a technical hypothesis needed to exclude pathological examples in characteristic $p>0$. We will review the relevant definitions in Section~\ref{section:Frobenius-classical}.

The lower bound in Theorem~\ref{main-theorem} may be improved or weakened under different hypotheses.
In Section~\ref{section:special_case}, we first prove the theorem in a special case, where the hypothesis is easy to state, and no knowledge of reflexivity is required.
Furthermore, the proof in the special case yields a sharper bound $q\geq c_0 d$ for some algebraic number $c_0\approx 1.537$, contains the key strategy to be reused in the other cases, and motivates the definition of the auxiliary surfaces introduced later in Section~\ref{subsect:auxillarySurfaces}.

We prove Theorem~\ref{main-theorem} in Section~\ref{section:general_case} under a slightly more general setup, namely, for \emph{Frobenius classical} surfaces.
We also prove a version of the theorem in Section~\ref{section:any-smooth-surface} for all smooth surfaces at the cost of a weaker bound $q\geq d^2$.
In Section~\ref{section:Frobenius-classical}, we show that reflexive surfaces are Frobenius classical; the results in this last section are valid for any hypersurface.

\medskip

The following example provides evidence that the condition $q\geq d-1$ is necessary to guarantee the existence of a transverse line; so our linear bound $q\geq cd$ is tight up to the multiplicative constant.

\begin{Eg}\label{counter-example}
Consider a surface $S\subseteq\bP^3$ defined by the polynomial
\begin{align*}
& L_{1}(X^q Y - X Y^q) + L_{2}(X^q Z - X Z^q) + L_{3}(X^q W - X W^q) \\ 
& + L_{4}(Y^q Z - Y Z^q) + L_{5}(Y^q W - Y W^q) + L_{6}(Z^q W - Z W^q)
\end{align*}
where $L_1, ..., L_6 \in \bF_q[X, Y, Z, W]$ are linear forms. 
This surface has degree $d = q+2$, and it is \emph{space-filling}, i.e. $S(\bF_q)=\bP^3(\bF_q)$. For each $\bF_q$-line $L\subseteq\bP^3$, either $L$ is contained in $S$, or $S\cap L$ contains $q+2$ points counted with multiplicity. In the latter case, $S\cap L$ already contains $q+1$ distinct point of $L(\bF_q)$ as $S$ is space-filling, so the extra intersection multiplicity accounts for tangency at one of the $\bF_q$-points. Thus, each line $L\subseteq\bP^3$ defined over $\bF_q$ is tangent to $S$ at some point.

In this example, one expects to get smooth surfaces by choosing $L_1, ..., L_6$ carefully. Indeed, computations in Macaulay2 \cite{M2} confirm the existence of such surfaces over $\bF_p$ for primes $p\leq 31$. However, we do not have a proof of this assertion in general. In any case, it produces a degree $d$ surface $S\subseteq \bP^3$ over $\bF_q$ satisfying $q=d-2$ such that $S$ admits no transverse $\bF_q$-lines.
\end{Eg}

\subsection*{Conventions}
We will assume that the characteristic of the field is $p\neq 2$. Some of the results, such as Theorem~\ref{Thm:special_case}, holds for $p=2$ but other concepts such as reflexivity is more delicate in this case.

\subsection*{Acknowledgements}
We thank Brendan Hassett for the support and valuable suggestions. We thank the referee for the comments on the manuscript. Research by the first author was partially supported by funds from NSF Grant DMS-1701659.

\section{Existence of transverse lines: special case}\label{section:special_case}

In this section, we prove Theorem~\ref{main-theorem} in a special case.
Let $S\subseteq\bP^3$ be a smooth surface defined by a degree $d$ homogeneous polynomial
\[
    F = F(X_0,X_1,X_2,X_3)\in\bF_q[X_0,X_1,X_2,X_3].
\]
For the sake of brevity,
we denote $F_{X_i} \colonequals \frac{\partial F}{\partial X_i}$.
Using the Frobenius morphism $\Phi\colon\bP^3\rightarrow\bP^3$,
which is defined as
\[
    \Phi\left([X_0,X_1,X_2,X_3]\right) = [X_0^q, X_1^q, X_2^q, X_3^q],
\]
we denote
\begin{align*}
    F_{X_i}^{(q)}(X_0, X_1, X_3, X_4)
    &\colonequals F_{X_i}\circ\Phi(X_0, X_1, X_3, X_4)\\
    &\;= F_{X_i}(X_0^q, X_1^q, X_2^q, X_3^q).
\end{align*}
With this notation,
we construct two polynomials from $F$ by
\begin{align*}
    & F_{1,0} \colonequals X_0^qF_{X_0} + X_1^qF_{X_1} + X_2^qF_{X_2} + X_3^qF_{X_3}\\
    & F_{0,1} \colonequals X_0F_{X_0}^{(q)} + X_1F_{X_1}^{(q)} + X_2F_{X_2}^{(q)} + X_3F_{X_3}^{(q)},
\end{align*}
and let
\[
    S_{1,0} \colonequals \{F_{1,0} = 0\}
    \quad\text{and}\quad
    S_{0,1} \colonequals \{F_{0,1} = 0\}.
\]
These surfaces are special cases of the \emph{auxiliary surfaces} associated to $S$ to be introduced in Section~\ref{subsect:auxillarySurfaces}.

\begin{Teo}\label{Thm:special_case}
Let $S\subseteq\bP^3$ be a smooth surface of degree $d$ defined over $\bF_q$, and let $S_{1,0}, S_{0,1}\subseteq\bP^3$ be the auxiliary surfaces defined as above.
Suppose that $S$, $S_{1,0}$, and $S_{0,1}$ intersect in a $0$-dimensional scheme.
Then there exists an $\bF_q$-line meeting $S$ transversely if $q\geq cd$,
where $c\approx1.536974$ is the real root of the polynomial $2x^3-2x^2-x-1$.
More precisely, without the assumption $q\geq cd$, the number of transverse $\bF_q$-lines is at least
\[
    q^4 - (d-2)q^3 - \frac{1}{2}\left[
    (d^2 - 5)q^2
    + (d^3 - 2d^2 + 4d - 4)q
    + (d^2 - 3)
    \right].
\]
\end{Teo}

\subsection{Main ingredients in the proof}

We prove the existence of an $\bF_q$-line transverse to $S$ by comparing the number of $\bF_q$-lines tangent to $S$ to the number of $\bF_q$-lines in $\bP^3$, which is given by 
\begin{equation}\label{numberOfFqLines}
    \#\operatorname{Gr}(2,4)(\bF_q) = (q^2+1)(q^2+q+1).
\end{equation}
Note that the latter quantity is a degree $4$ polynomial in $q$.
We will show that the previous one grows at a rate no greater than a degree $3$ polynomial in $q$.
Therefore, we can find a transverse line when $q$ is large enough compared to $\deg(S)$, and the dependency between $q$ and $\deg(S)$ can be analyzed by comparing the two polynomials.

Recall that a line $L\subseteq\bP^3$ is tangent to $S$ at $P$ if and only if
\[
	P\in L\subseteq T_PS.
\]
Here we consider $T_PS$ as a hyperplane in the ambient $\bP^3$.
In order to estimate the number of $\bF_q$-lines tangent to $S$,
we divide them into two different types.

\begin{Def}\label{def:rationalSpecialTangentLines}
Let $L\subseteq\bP^3$ be an $\bF_q$-line tangent to $S$.
We call $L$ a \emph{rational tangent line} if it is tangent to $S$ at some $\bF_q$-point $P\in S$.
Otherwise, we call $L$ a \emph{special tangent line}.
\end{Def}

Estimating the number of special tangent lines is subtle, which will be investigated in Section~\ref{subsect:estimateSpTan_Special}.
On the other hand, the number of rational tangent lines is easy to estimate.
Indeed, let $L$ be an $\bF_q$-line tangent to $S$ at $P\in S(\bF_q)$.
Then $L$ must be one of the $q+1$ lines in $T_PS\cong\bP^2$ defined over $\bF_q$ and passing through $P$.
Therefore, the total number of rational tangent lines is bounded above as follows:
\begin{equation}\label{estimateRats}
    \#\{\text{rational tangents}\}
    \leq \#S(\bF_q)\cdot(q+1).
\end{equation}

Another ingredient in the proof is the bound on $\#S(\bF_q)$ given by Homma~\cite[Theorem~1.1]{Homma-lines}
\begin{equation}\label{HommaBound}
    \#S(\bF_q) \leq (d-1)(q^2+1),
\end{equation}
which holds whenever \emph{$S$ contains no $\bF_q$-line}.
In our situation, there is no $\bF_q$-line in $S$ since any such line is contained in both $S_{1,0}$ and $S_{0,1}$ by Lemma~\ref{Lemma:FqLineInAuxillary}.
But this is not allowed as $S\cap S_{1, 0}\cap S_{0, 1}$ is a $0$-dimensional scheme by hypothesis.

\subsection{Estimate for the number of special tangent lines}
\label{subsect:estimateSpTan_Special}

From the definitions of $S_{1,0}$ and $S_{0,1}$, it is straightforward to verify that
\begin{equation}\label{geometricAuxiliary}
\begin{aligned}
    (S\cap S_{1,0})(\overline{\bF_q}) &=
    \left\{
        P\in S(\overline{\bF_{q}}): \Phi(P)\in T_{P}S
    \right\},\\
    (S\cap S_{0,1})(\overline{\bF_q}) &=
    \left\{
        P\in S(\overline{\bF_{q}}): P\in T_{\Phi(P)}S
    \right\}.
\end{aligned}
\end{equation}
Consider the intersection
\[
    \Gamma \colonequals S\cap S_{1,0}\cap S_{0,1}.
\]
Then $\Gamma(\overline{\bF_q})$ is a finite set by our hypothesis. Note that $\Gamma(\bF_q) = S(\bF_q)$.

\begin{Lemma}\label{Lemma:spTanToTangency_Special}
For each special tangent line $L\subseteq\bP^3$,
we define
\[
    \mathcal{P}_{L} \colonequals
    \left\{
        P\in S(\overline{\bF_{q}}) : P\in L\subseteq T_{P}S
    \right\}
\]
to be the set of the points of tangency.
Then
\begin{enumerate}
    \item\label{spTanToTangency:inGamma}
    $\mathcal{P}_{L}\subseteq\Gamma(\overline{\bF_q})\setminus \Gamma(\bF_q)$, and
    \item\label{spTanToTangency:disjoint}
    $\mathcal{P}_{L}\cap\mathcal{P}_{L'}\neq\emptyset$ implies $L = L'$.
\end{enumerate}
\end{Lemma}
\begin{proof}
Given any point $P\in\mathcal{P}_{L}$,
we claim that $P\in\Gamma(\overline{\bF_q})\setminus\Gamma(\bF_q)$.
Indeed, since $P$ is a point of tangency, we have
\begin{equation}\label{tangentToP}
    P\in L\subseteq T_PS.
\end{equation}
Under the Frobenious action $\Phi$, there are relations
\[
    \Phi(L) = L
    \quad\text{and}\quad
    \Phi(T_PS) = T_{\Phi(P)}S,
\]
so we obtain
\[
    \Phi^r(P)\in L\subseteq T_{\Phi^r(P)}S
    \;\text{ for all }r\in\bZ
\]
by applying $\Phi^{r}$ to (\ref{tangentToP}).
In particular, we get
\[
    P\in L\subseteq T_{\Phi(P)}S
    \quad\text{and}\quad
    \Phi(P)\in L\subseteq T_PS,
\]
which imply that $P\in S_{1,0}(\overline{\bF_q})$ and $P\in S_{0,1}(\overline{\bF_q})$, respectively.
It follows that $P\in\Gamma(\overline{\bF_q})$.
Note that $P\notin\Gamma(\bF_q) = S(\bF_q)$ as $L$ is a special tangent line.
Hence (\ref{spTanToTangency:inGamma}) follows.

Let $L$ and $L'$ be two special tangent lines.
Suppose that there exists a point $P\in\mathcal{P}_{L}\cap\mathcal{P}_{L'}$.
Using the fact that $L$ and $L'$ are defined over $\bF_q$, we have
\[
    \Phi(P)\in\Phi(L)\cap \Phi(L')=L\cap L',
\]
so the distinct points $P$ and $\Phi(P)$ both lie in $L$ and $L'$, which implies that $L=L'$.
\end{proof}

The assignment $L\mapsto\mathcal{P}_{L}$ shows that every special tangent line $L$ contributes at least $\#\mathcal{P}_L \geq 2$ distinct points to $\Gamma(\overline{\bF_q})\setminus\Gamma(\bF_q)$.
Hence we obtain the following inequalities:
\begin{align*}
    \#\{\text{special tangents}\}&\leq\frac{1}{2}\left(\#\Gamma(\overline{\bF_q}) - \#\Gamma(\bF_q)\right) \\
    &= \frac{1}{2}\left(\#\Gamma(\overline{\bF_q}) - \#S(\bF_q)\right).
\end{align*}
On the other hand, the definition of $\Gamma$ implies that
\begin{align*}
    \#\Gamma(\overline{\bF_q}) &\leq
    \deg S\cdot\deg S_{1,0}\cdot\deg S_{0,1}\\
    &= d(q+d-1)(qd-q+1).
\end{align*}
As a result, we obtain
\begin{equation}\label{estimateSpTanExplicit}
    \#\{\text{special tangents}\}
    \leq\frac{d(q+d-1)(qd-q+1) - \#S(\bF_q)}{2}.
\end{equation}

\subsection{Estimate for the number of transverse lines}
\label{subsect:existTransLine_Special}

Inequalities (\ref{estimateRats}) and (\ref{estimateSpTanExplicit}) together imply that
\begin{align*}
    &\#\{\text{tangent lines}\} = 
    \#\{\text{rational tangents}\} + \#\{\text{special tangents}\} \\
    &\leq \# S(\bF_q)\cdot (q+1) + \frac{d(q+d-1)(qd-q+1)-\# S(\bF_q)}{2}\\
    &\leq \# S(\bF_q)\cdot\left(q+\frac{1}{2}\right) + \frac{d(q+d-1)(qd-q+1)}{2}\\
    &\leq(d-1)(q^2+1)\left(q+\frac{1}{2}\right) + \frac{d(q+d-1)(qd-q+1)}{2},
\end{align*}
where the last inequality uses the bound (\ref{HommaBound}) for $\# S(\bF_q)$.

Recall from (\ref{numberOfFqLines}) that the number of $\bF_q$-lines in $\bP^3$ equals $(q^2+1)(q^2+q+1)$.
Hence
\begin{align*}
    &\#\{\text{transverse lines}\}
    = (q^2+1)(q^2+q+1) - \#\{\text{tangent lines}\}\\
    &\geq (q^2+1)(q^2+q+1)
    - (d-1)(q^2+1)\left(q+\frac{1}{2}\right)\\
    &\hspace{12pt} - \frac{d(q+d-1)(qd-q+1)}{2}\\
    &= q^4 - (d-2)q^3 - \frac{1}{2}\left[
    (d^2 - 5)q^2
    + (d^3 - 2d^2 + 4d - 4)q
    + (d^2 - 3)
    \right].
\end{align*}
The last expression can be considered as a polynomial in $q$ with positive leading coefficient.
In particular, there exists an $\bF_q$-line transverse to $S$ if $q$ is large enough compared to $d$.

Our goal is to find minimal such $q$ of the form $cd$, where $c$ is a real constant.
By substituting $q = xd$ into the polynomial and requiring it to be positive, we get the inequality
\begin{align*}
    &x (2x^3 - 2x^2 - x - 1)d^4\\
    &+ (4x^3 + 2x)d^3
    + (5x^2-4x-1)d^2
    + 4xd + 3 > 0.
\end{align*}
Now consider the left hand side as a polynomial in $d$.
To make the inequality hold for all $d\in\bN$,
it is necessary that
\[
    x(2x^3-2x^2-x-1) \geq 0.
\]
The minimal $x$ where this relation is satisfied is $x = c$ where $c$ is the unique real root of $2x^3-2x^2-x-1$, namely
\[
    c = \frac{1}{6}\left(2+\sqrt[3]{80-30\sqrt{6}} + \sqrt[3]{80+30\sqrt{6}} \right) = 1.536974....
\]
Moreover, it is straightforward to verify that all the other coefficients
\[
    4x^3 + 2x,\quad
    5x^2 - 4x - 1,\quad
    \text{and }4x
\]
are strictly positive when $x\geq c$.
Therefore, to satisfy the inequality above, it is sufficient to have $q\geq cd$.

\section{Transverse lines to Frobenius classical surfaces} \label{section:general_case}

In this section, we prove an analogue of Theorem~\ref{Thm:special_case} which is slightly weaker but deals with a more general situation. Here we retain the notation from the beginning of Section~\ref{section:special_case}.

\begin{Def}\label{def:frobeniusClassical}
Let $S\subseteq\bP^3$ be a smooth surface defined over $\bF_q$.
We say that $S$ is \emph{Frobenius classical} if there exists a closed point $P\in S$ such that $\Phi(P)\notin T_{P}S$.
Otherwise, $S$ is called \emph{Fronenius non-classical}.
\end{Def}

We discuss Frobenius classical surfaces later in more detail in Section~\ref{section:Frobenius-classical}. In particular, we will show that every reflexive surface is Frobenius classical, and therefore Theorem~\ref{main-theorem} is a consequence of Theorem~\ref{Thm:general_case}.

\begin{Teo}\label{Thm:general_case}
Let $S\subseteq\bP^3$ be a smooth surface defined over $\bF_q$.
Assume that $S$ is Frobenius classical.
Then there exists an $\bF_q$-line transverse to $S$ if
\[
    q\geq\left(\frac{3+\sqrt{17}}{4}\right)d \approx 1.7808d
\]
More precisely, under the assumption $q\geq d$, the number of transverse $\bF_q$-lines is at least
\[
    q^4 - \frac{1}{2}\left[
    (3d-4)q^3 + (d^2 + 3d - 6)q^2
    + d(d + 1)q + d(d-1)^2
    \right].
\]
\end{Teo}

Let us explain why the theorem deals with a more general situation than in Theorem~\ref{Thm:special_case}:
By definition, a surface $S$ is Frobenius non-classical if and only if
\[
  \Phi(P)\in T_PS,\quad\text{for all }P\in S(\overline{\bF_q}).  
\]
In view of (\ref{geometricAuxiliary}), this is equivalent to $S\subseteq S_{1,0}$,
which implies that the intersection $S\cap S_{1,0}\cap S_{0,1}$ is at least $1$-dimensional.
Therefore, $S$ is Frobenius classical if we require the intersection to be $0$-dimensional.

\subsection{Main ingredients in the proof}\label{subsection-general-proof}

The strategy in proving Theorem~\ref{Thm:general_case} is the same as in the special case, except that the estimate for the number of special tangent lines involves the additional auxiliary surface $S_{2,0}$, defined by
\[
    X_0^{q^2}F_{X_0} + X_1^{q^2}F_{X_1}
    + X_2^{q^2}F_{X_2} + X_3^{q^2}F_{X_3} = 0.
\]
Note that
\[
    (S\cap S_{2,0})(\overline{\bF_q}) =
    \left\{
        P\in S(\overline{\bF_q}) : \Phi^2(P)\in T_PS
    \right\}.
\]

Since $S$ is Frobenius classical,
$S\cap S_{1,0}$ is $1$-dimensional.
Therefore the intersection
\[
    \Pi \colonequals S\cap S_{1,0}\cap S_{0,1}\cap S_{2,0}
\]
has no component in dimension two,
which allows us to write
\[
    \Pi = \Pi_0\cup\Pi_1
\]
where $\dim\Pi_0 = 0$ and $\dim\Pi_1 = 1$.
We show in Lemma~\ref{Lemma:weird_curve} that $\Pi_1$ is a union of $\bF_q$-lines.
This fact implies that the points of tangency of a special tangent line must lie in $\Pi_0$,
which helps us produce an upper bound to the number of these lines.
The details are in Section~\ref{subsect:estimateSpTan_General}.

In this case, it is possible that $S$ contains an $\bF_q$-line, so the bound (\ref{HommaBound}) for $\#S(\bF_q)$ shall be replaced by \cite[Theorem~1.2]{Homma-Kim}:
\begin{equation}\label{lemma:Homma-Kim}
    \#S(\bF_q)\leq (q+1)(qd-q+1).
\end{equation}

\subsection{Collinearity on Galois conjugates}
\label{subsect:collinearGalois}

The goal of this part is Lemma~\ref{Lemma:weird_curve},
which shows that the component $\Pi_1$ consists only of $\bF_q$-lines.
In the following, we use the notation $\left<P_1,...,P_k\right>$ to denote the subspace in $\bP^3$ spanned by the points $P_1,...,P_k\in\bP^3$.

\begin{Lemma}\label{lemma:2pts-line}
Let $P\in\bP^3$ be a point.
Then $\left<P, \Phi(P), \Phi^2(P)\right>$ is a line if and only if $\left<\Phi^r(P): r\in\bZ\right>$ is a line. In this situation, the two lines coincide and are defined over $\bF_q$.
\end{Lemma}
\begin{proof}
If $\left<\Phi^r(P): r\in\bZ\right>$ is a line,
then of course $\left<P, \Phi(P), \Phi^2(P)\right>$ is a line.
For the converse,
the statement is trivial when $P$ is defined over an extension $\bF_{q^s}$ of degree $s\leq 3$,
so we assume this is not the case.
Applying the Frobenious action to $L = \left<P, \Phi(P), \Phi^2(P)\right>$, we get
\begin{gather*}
    \Phi(L)
    = \left<\Phi(P), \Phi^2(P), \Phi^3(P)\right>\\
    = \left<\Phi(P), \Phi^2(P)\right>
    = \left<P, \Phi(P), \Phi^2(P)\right>
    = L
\end{gather*}
as two distinct points uniquely determine a line.
Therefore $L$ is defined over $\bF_q$, and thus contains $\Phi^r(P)$ for all $r\in\bZ$.
\end{proof}

\begin{Lemma}\label{lemma:3pts-plane}
Let $P\in\Pi(\overline{\bF_q})$ be any given point.
Then
\[
    \Phi^r(P)\in T_P S\text{ for all }r\in\bZ.
\]
Assume further that $\left\{\Phi^r(P):r\in\bZ\right\}$ is not contained in a line.
Then
\[
    T_PS = T_{\Phi^r(P)}S
    \;\text{ for all }r\in\bZ.
\]
\end{Lemma}
\begin{proof}
By definition of $\Pi$, we have
\[
    \Phi^{-1}(P),\; P,\; \Phi(P),\; \Phi^2(P)\in T_PS.
\]
If three consecutive points from the above four are collinear,
the lemma follows immediately from Lemma~\ref{lemma:2pts-line}.

Assume that $\Phi^{-1}(P)$, $P$, and $\Phi(P)$ are not collinear.
Then $P$, $\Phi(P)$, $\Phi^2(P)$ are also not collinear. As three non-collinear points uniquely determine a plane, we get
\[
    \left<\Phi^{-1}(P), P, \Phi(P)\right>
    = T_P S =
    \left<P, \Phi(P), \Phi^2(P)\right>
\]
As a result,
\begin{align*}
    \Phi(T_PS) = \Phi(\left<\Phi^{-1}(P), P, \Phi(P)\right>) = \left<P, \Phi(P), \Phi^2(P)\right>
    = T_P S.
\end{align*}
Thus, $T_P S$ is defined over $\bF_q$. So, $\Phi^{r}(T_P S)=T_P S$ for all $r\in\bZ$ which translates into $T_{\Phi^{r}(P)} S=T_{P} S$. In particular, $\Phi^{r}(P)\in T_{P}S$ for all $r\in\bZ$.
\end{proof}

\begin{Lemma}\label{Lemma:weird_curve}
The component $\Pi_1$ is a union of $\bF_q$-lines.
\end{Lemma}
\begin{proof}
Let $C\subseteq\Pi_1$ be a component defined and irreducible over $\bF_q$.
Assume that $C$ is not an $\bF_q$-line.
Pick a point $P\in C$ which is defined over $\bF_{q^n}$ for some $n>\deg C$ but not over any proper subfield.

Assume that $\Phi^{-1}(P), P, \Phi(P)$ are collinear. Then they span a line $L$ defined over $\bF_q$ by Lemma~\ref{lemma:2pts-line}.
Moreover, the intersection $L\cap C$ contains the set
\[
    \left\{
        \Phi^r(P) : r = 0,...,n-1
    \right\}
\]
with cardinality $n > \deg C$.
This implies $L = C$, a contradiction.

Therefore, $\Phi^{-1}(P), P, \Phi(P)$ cannot be collinear.
In this situation, all of the tangent planes $T_{\Phi^r(P)}S$ coincide by Lemma~\ref{lemma:3pts-plane}.
It follows that the \emph{Gauss map}
\[
    \gamma\colon S\rightarrow(\bP^3)^{\ast}
    : Q\mapsto T_QS
\]
sends every $\Phi^r(P)\in S$ to the same point.
Here $(\bP^3)^{\ast}$ denotes the space of hyperplanes in $\bP^3$.
Since the point $P\in C(\bF_{q^n})$ can be chosen with $n$ arbitrarily high,
$\gamma$ must contract $C$. This contradicts the fact that the Gauss map of a smooth surface is finite \cite[Corollary~I.2.8]{Zak}.
\end{proof}

\subsection{Estimate for the number of special tangent lines}
\label{subsect:estimateSpTan_General}

\begin{Lemma}\label{Lemma:spTanToTangency_General}
For each special tangent line $L\subseteq\bP^3$,
let
\[
    \mathcal{P}_{L} \colonequals
    \left\{
        P\in S(\overline{\bF_{q}}) : P\in L\subseteq T_{P}S
    \right\}
\]
be the set of the points of tangency.
Then
\begin{enumerate}
    \item\label{spTanToTangency:inPi}
    $\mathcal{P}_{L}\subseteq\Pi_0(\overline{\bF_q})\setminus\Pi_0(\bF_q)$, and
    \item\label{spTanToTangency:disjoint_General}
    $\mathcal{P}_{L}\cap\mathcal{P}_{L'}\neq\emptyset$ implies that $L = L'$.
\end{enumerate}
\end{Lemma}
\begin{proof}
A similar argument as in the proof of Lemma~\ref{Lemma:spTanToTangency_Special} shows that
\begin{enumerate}[label=(\roman*)]
    \item $\mathcal{P}_{L}\subseteq\Pi(\overline{\bF_q})\setminus\Pi(\bF_q)$, and
    \item\label{spTanToTangency:disjoint_inPf}
    $\mathcal{P}_{L}\cap\mathcal{P}_{L'}\neq\emptyset$ implies that $L = L'$.
\end{enumerate}
This already proves (\ref{spTanToTangency:disjoint_General}).

To prove (\ref{spTanToTangency:inPi}), assume that there exists a point $P\in\mathcal{P}_{L}\cap\Pi_1(\overline{\bF_q})$.
Since $\Pi_1$ is a union of $\bF_q$-lines by Lemma~\ref{Lemma:weird_curve},
we have $P\in L'$ for some $\bF_q$-line
\[
    L'\subseteq\Pi_1\subseteq S.
\]
Then both $L$ and $L'$ contains the distinct points $P$ and $\Phi(P)$.
It follows that $L = L'\subseteq S$, a contradiction.
Therefore $\mathcal{P}_{L}$ is disjoint from $\Pi_1(\overline{\bF_q})$, which proves (\ref{spTanToTangency:inPi}).
\end{proof}

The assignment $L\mapsto\mathcal{P}_{L}$ shows that every special tangent line $L$ contributes at least $\#\mathcal{P}_L \geq 2$ distinct points to $\Pi_0(\overline{\bF_q})$.
Hence we obtain the following inequality:
\[
    \#\{\text{special tangents}\}
    \leq\frac{\#\Pi_0(\overline{\bF_q})}{2}.
\]

To estimate $\#\Pi_0(\overline{\bF_q})$, consider the $1$-dimensional scheme
\[
    B := (S\cap S_{1,0})\setminus\Pi_1.
\]
It decomposes into geometrically irreducible components as
\[
    B = B_1\cup\cdots\cup B_m.
\]
Hence
\[
    \Pi_0 = B\cap(S_{0,1}\cap S_{2,0})
    = \bigcup _{i=1}^m B_i\cap (S_{0,1}\cap S_{2,0}),
\]
whence
\begin{align*}
    \#\Pi_0(\overline{\bF_q})
    &\leq \sum_{i=1}^m
    \#(B_i\cap S_{0,1}\cap S_{2,0})(\overline{\bF_q})\\
    &\leq\sum_{i=1}^m\deg B_i\cdot
    \max\left\{\deg S_{0,1}, \;\deg S_{2,0}\right\}\\
    &= \deg B\cdot
    \max\left\{\deg S_{0,1}, \;\deg S_{2,0}\right\}\\
    &\leq d(q+d-1)\cdot
    \max\{dq - q + 1, \;q^2 + d - 1\},
\end{align*}
where in the second inequality we use the fact that $B_i\cap S_{0,1}$ or $B_i\cap S_{2,0}$ must be of dimension $0$.

Assume that $q\geq d$. Then it is easy to verify that
\[
    \max\{dq - q + 1, \;q^2 + d - 1\} = q^2 + d - 1.
\]
As a result, we obtain
\begin{equation}\label{estimateSpTanExplicit_General}
\begin{aligned}
    \#\{\text{special tangents}\}
    &\leq\frac{\#\Pi_0(\overline{\bF_q})}{2}\\
    &\leq\frac{d(q+d-1)(q^2 + d - 1)}{2}.
\end{aligned}
\end{equation}

\subsection{Estimate for the number of transverse lines}
\label{subsect:existTransLine_General}

Note that the estimate (\ref{estimateRats}) for the number of rational tangents is still valid in the general case:
\[
    \#\{\text{rational tangents}\}
    \leq\#S(\bF_q)\cdot(q+1).
\]
Together with (\ref{lemma:Homma-Kim}) and (\ref{estimateSpTanExplicit_General}), we obtain
\begin{align*}
    \#\{\text{tangent lines}\}
    &= \#\{\text{rational tangents}\}
    + \#\{\text{special tangents}\} \\
    &\leq\# S(\bF_q)\cdot (q+1) + \frac{d(q+d-1)(q^2+d-1)}{2}\\
    &\leq (q+1)^2(qd-q+1) + \frac{d(q+d-1)(q^2+d-1)}{2}.
\end{align*}
According to (\ref{numberOfFqLines}), 
\begin{align*}
    &\#\{\text{transverse lines}\}
    = (q^2+1)(q^2+q+1) - \#\{\text{tangent lines}\}\\
    &\geq (q^2+1)(q^2+q+1)
    - (q+1)^2(qd-q+1)\\
    &\hspace{12pt} - \frac{d(q+d-1)(q^2+d-1)}{2}\\
    &= q^4 - \frac{1}{2}\left[
    (3d-4)q^3 + (d^2 + 3d - 6)q^2
    + d(d + 1)q + d(d-1)^2
    \right].
\end{align*}

Similar to the special case,
our goal here is to find a constant $c>0$ such that the last expression is positive whenever $q = xd$ with $x\geq c$.
By inserting $q = xd$ into the expression followed by a rearrangement,
we want to find minimal $x$ such that
\begin{align*}
    &x^2(2x^2 - 3x - 1)d^4\\
    &+ (4x^3 - 3x^2 - x - 1)d^3
    + (6x^2 - x + 2)d^2
    - d > 0.
\end{align*}
For this inequality to hold for all $d\in\bN$, it is necessary that
\[
    2x^2 - 3x - 1\geq 0.
\]
The minimal value of $x>0$ for which this relation holds is when $x=c$ is the unique positive root of $2x^2-3x-1=0$, namely 
$$
  c = \frac{3+\sqrt{17}}{4}  = 1.780776...
$$   
Furthermore, the coefficient $4x^3-3x^2-x-1$ in front of $d^3$, and the entire expression $(6x^2-x+2)d^2-d$ are also strictly positive for $x\geq c$.  Indeed, the unique real root of $4x^3-3x^2-x-1=0$ is at $x\approx 1.1542$, and 
\[
(6x^2-x+2)d^2-d \geq 5x^2 d^2 - d \geq 5d^2 - d > 0 
\]
for all positive integers $d$. We deduce the existence of a transverse $\bF_q$-line provided that $q\geq cd$ where $c = \frac{3+\sqrt{17}}{4}$.

\subsection{Auxiliary surfaces}
\label{subsect:auxillarySurfaces}

The surfaces $S_{1,0}$, $S_{0,1}$, and $S_{2,0}$ play essential roles in proving our main theorems. In this subsection we define $S_{m, n}$ for any $m, n\in\bZ$, and prove that any $\bF_q$-line inside $S$ is contained in $S_{m, n}$. The inspiration for these surfaces comes from the work of St\"ohr and Voloch \cite[Theorem 0.1]{Stohr-Voloch}, where they used the exact analogue of $S_{1, 0}$ for plane curves to give an upper bound on the number of $\bF_q$-points.

Suppose that $S$ is a smooth surface defined by the equation $F = 0$ over $\bF_q$.
Consider a point
\[
    P = [X_0 : X_1 : X_2 : X_3]\in\bP^3
\]
as a row vector, and the differential
\[
    DF = \left[F_{X_0} : F_{X_1} : F_{X_2} : F_{X_3}\right],
    \quad\text{where }
    F_{X_i} := \frac{\partial F}{\partial X_i},
\]
as a column vector. We define the surfaces
\[
    S_{m,n} \colonequals \left\{
    \Phi^m(P) \cdot DF(\Phi^n(P)) = 0
    \right\},
    \quad m, n\in\bZ,
\]
where the products are inner products between row and column vectors.
These are called \emph{auxiliary surfaces} throughout the paper. Explicitly, $S_{m, n}$ is defined by the equation
\[
    X_0^{q^{m}}F_{X_0}(X_0^{q^{n}}, X_1^{q^{n}}, X_2^{q^n}, X_3^{q^n})
    + ... +
    X_3^{q^{m}}F_{X_3}(X_0^{q^{n}}, X_1^{q^{n}}, X_2^{q^n}, X_3^{q^n}) = 0.
\]

By definition, 
\[
    (S\cap S_{m,n})(\overline{\bF_q}) = \left\{
    P\in S(\overline{\bF_q}) :
    \Phi^m(P)\in T_{\Phi^n(P)}S
    \right\}.
\]
Note that, as $F$ is defined over the ground field, the equation for $S_{m,n}$ is equivalent to
\[
    \Phi^m(P) \cdot \Phi^n(DF(P)) = 0,
\]
from which one can easily verify that, set-theoretically,
\[
    S_{m,n} \stackrel{\text{set}}{=} S_{m+r, n+r}
    \;\text{ for all }r\in\bZ.
\]

\begin{Lemma}\label{Lemma:FqLineInAuxillary}
Assume that $S$ contains an $\bF_q$-line $L$.
Then $L\subseteq S_{m,n}$ for all $m,n\in\bZ$.
\end{Lemma}
\begin{proof}
Clearly, $L\subseteq T_PS$ for all $P\in L$.
Since $L$ is defined over $\bF_q$, given any point $P\in L$, we have
\[
    \Phi^m(P) \in L \subseteq T_{\Phi^n(P)}S,
    \quad\text{for all }m,n\in\bZ.
\]
So $P\in S_{m, n}$ for all $m,n\in\bZ$.
\end{proof}

\begin{Rmk}
The equation for $S_{1,0}$ may possibly be the zero polynomial. Actually, the surface $S$ in the introduction
\[
    X_0^qX_1 - X_0X_1^q + X_2^qX_3 - X_2X_3^q = 0
\]
gives such an example. On the other hand, the equation for $S_{0,2}$ is always nonzero. In fact, for a hypersurface in $\bP^n$ defined by the equation $F(X_0,\dots,X_n) = 0$, we have
\[
    \deg_{X_i}X_jF_{X_j}(X_0^{q^n},\dots,X_n^{q^n}) \equiv\delta_{ij} (\text{mod }q),
\]
which implies that the polynomial $\sum_jX_jF_{X_j}(X_0^{q^n},\dots,X_n^{q^n})$ is nonzero.
\end{Rmk}

\section{Transverse lines to arbitrary smooth surfaces}\label{section:any-smooth-surface}

In the previous section, we produced transverse lines to a Frobenius-classical surface $S\subseteq\bP^3$ of degree $d$ under the assumption that $q\geq cd$ for some constant $c>0$. Now we investigate the general case of a smooth surface in $\bP^3$ without any additional hypothesis. 

We first explain why the bound $q\geq d(d-1)^2$ is sufficient to guarantee a transverse $\bF_q$-line to an arbitrary smooth surface of degree $d$. Under the hypothesis $q\geq d(d-1)^2$, Ballico \cite[Theorem 1]{Bal} proves the existence of a transverse $\bF_q$-plane $H\subseteq \bP^3$ to $S$, meaning that $T_P X\not\subseteq H$ for every $P\in S(\overline{\bF_q})$. Note that $C\colonequals S\cap H$ is a smooth curve. Applying Ballico's result again to $C$, we obtain an $\bF_q$-line $L\subseteq H$ such that $L$ is transverse to $C$. By construction, $L$ is also transverse to $S$.

The purpose of this section is to improve the bound $q\geq d(d-1)^2$ to a quadratic bound $q\geq d^2$. 

\begin{Teo}\label{Teo:any-smooth-surface}
Let $S\subseteq \bP^3$ be a smooth surface of degree $d$ defined over $\bF_q$. If $q\geq d^2$, then there exists an $\bF_q$-line $L\subseteq\bP^3$ meeting $S$ transversely.
More precisely, under the assumption $q\geq d$, the number of transverse $\bF_q$-lines is at least
\[
    q^4 - \frac{1}{2}\left[
    (d^2 + d - 4)q^3
    + (5d - 6)q^2
    + d(d^2 - 2d + 3)q
    + d(d-1)
    \right].
\]
\end{Teo}

The advantage of the theorem is that it applies to every smooth surface without the additional hypothesis that $S$ is reflexive or Frobenius classical. As a drawback, we get a quadratic bound $q\geq d^2$ as opposed to a linear bound $q\geq cd$.  

The key in proving Theorem~\ref{Teo:any-smooth-surface} is to show that either $S\cap S_{1, 0}$ or $S\cap S_{2, 0}$ must be a curve. Recall that if $S=\{F=0\}$, then $S_{1,0}$ and $S_{2,0}$ are auxiliary surfaces defined respectively by the equations
\begin{gather*}
X_0^q F_{X_0}(X_0, X_1, X_2, X_3) + \cdots + X_3^q F_{X_3}(X_0, X_1, X_2, X_3) = 0,\\
X_0^{q^2} F_{X_0}(X_0, X_1, X_2, X_3) + \cdots + X_3^{q^2} F_{X_3}(X_0, X_1, X_2, X_3) = 0.
\end{gather*}
We may assume $d>1$ as the case $d=1$ corresponds to $S$ being a plane which already admits plenty of transverse lines.

\subsection{Proper intersections with auxiliary surfaces}
\label{subsect:properIntersection}

\begin{Lemma}\label{prop:any-smooth-surface}
Let $S\subseteq \bP^3$ be a smooth surface of degree $d>1$ defined over $\bF_q$, where $q\geq d$.
Then $S\cap S_{1, 0}$ or $S\cap S_{2, 0}$ is $1$-dimensional.
\end{Lemma}

Recall that $S$ is Frobenius classical \emph{for} $\bF_q$ if and only if $S\cap S_{1, 0}$ is a curve. Therefore, an equivalent formulation of in Lemma~\ref{prop:any-smooth-surface} (2) is that $S$ is Frobenius classical for $\bF_{q}$, or Frobenius classical for $\bF_{q^2}$.

\begin{proof}[Proof of Lemma~\ref{prop:any-smooth-surface}]
Assume, to the contrary, that $S\cap S_{1, 0}$ and $S\cap S_{2, 0}$ are both surfaces.
Since $S$ is irreducible, we get that $S\subseteq S_{1, 0}$ and $S\subseteq S_{2, 0}$. By Proposition~\ref{good-hyperplane-sections} below, we can find an $\bF_q$-plane $H$ such that the plane curve $C\colonequals S\cap H$ has a component $D$ defined and irreducible over $\bF_q$ of degree $d>1$. 

Let $P\in D$ be a closed point. If $P, \Phi(P), \Phi^2(P)$ are collinear, then the line $L=\langle P, \Phi(P), \Phi^2(P)\rangle $ is an $\bF_q$-line, because $\Phi(L)=\langle \Phi(P), \Phi^2(P), \Phi^3(P)\rangle = L$. We deduce that:
\[
\{P\in D: \ P, \Phi(P), \Phi^2(P) \text{ are collinear}\} \subseteq \bigcup_{\bF_q \text{-lines } L} (L\cap D).
\]
As there are only finitely many $\bF_q$-lines, the set on the left hand side is finite.  Thus, for a general point $P$ on $D$, the points $P, \Phi(P), \Phi^{2}(P)$ are non-collinear. We have:
\begin{align}
    P, \Phi(P), \Phi^2(P) &\in H  \label{three-points-on-H} \\
    P, \Phi(P), \Phi^2(P) &\in T_{P} S. \label{three-points-on-T_QS}
\end{align}
The relation (\ref{three-points-on-H}) follows from $P\in H$ and the fact that $H$ is defined over $\bF_q$. The relation (\ref{three-points-on-T_QS}) follows from the assumption that $S\subseteq S_{1, 0}$ and $S\subseteq S_{2, 0}$. As $P, \Phi(P), \Phi^2(P)$ are non-collinear, we deduce that $H=T_{P} S$. We have shown that a general point $P\in D$ admits the same tangent plane to $S$, namely $H$. Consequently, the Gauss map $\gamma: S\rightarrow (\bP^3)^{\ast}$ contracts $D$. This contradicts Zak's result that the Gauss map of a smooth surface is finite \cite[Corollary~I.2.8]{Zak}. More details on the Gauss map can be found in Section~\ref{subsect:pre-reflexive}.
\end{proof}

\begin{Prop}\label{good-hyperplane-sections}
Let $S\subseteq \bP^3$ be a smooth surface of degree $d>1$ defined over $\bF_q$.
Assume that $q\geq d$. Then there exists an $\bF_q$-plane $H\subseteq\bP^3$ such that the plane curve $H\cap S$ contains a component $D$ defined and irreducible over $\bF_q$ with $\deg(D)>1$.
\end{Prop}

\begin{proof}
Assume, to the contrary, that $H\cap S$ is a union of $\bF_q$-lines for every $\bF_q$-plane $H$ in $\bP^3$. In this case, $H\cap S$ consists of $d$ \emph{distinct} $\bF_q$-lines, because a hyperplane section of a 
non-degenerate smooth surface in $\bP^3$ cannot have a non-reduced component of positive dimension. This follows from a general fact that if $X\subseteq\bP^{n}$ is a smooth hypersurface, then $X\cap H$ has isolated singularities for every hyperplane $H$. This is simply a restatement of Zak's theorem that the Gauss map (see Section~\ref{subsect:pre-reflexive}) is a finite morphism \cite[Corollary~I.2.8]{Zak}.

Given an $\bF_q$-plane $H\subseteq\bP^3$, write $H\cap S = E_1\cup \cdots \cup E_d$ where $E_i$ are distinct $\bF_q$-lines. Since $d\geq 2$, $H\cap S$ is singular at some $\bF_q$-point $Q$, which means that $H=T_{Q} S$. In particular, every $\bF_q$-plane $H$ is tangent; therefore, the Gauss map 
\[
\gamma\colon S\to (\bP^{3})^{\ast} \; : \; P\mapsto T_P S
\]
is surjective at the level of $\bF_q$-points: 
\[
S(\bF_q) \twoheadrightarrow (\bP^3)^{\ast}(\bF_q).
\]
Comparing the cardinalities,
\[
q^3 + q^2 + q + 1 = \# (\bP^3)^{\ast}(\bF_q) \leq \# S(\bF_q) \leq (q+1)(qd-q+1)
\]
where the right-most inequality is the bound (\ref{lemma:Homma-Kim}) due to Homma and Kim. Using the identity $q^3+q^2+q+1=(q+1)(q^2+1)$, the inequality above is equivalent to:
\[
q^2 + 1 \leq qd-q+1 \ \Leftrightarrow \ q^2\leq qd-q \ \Leftrightarrow \ q\leq d-1
\]
contradicting the initial hypothesis that $q\geq d$.
\end{proof}

\begin{Eg}\label{katz-Example}
The conclusion of Proposition~\ref{good-hyperplane-sections} does not hold if the condition $q\geq d$ is removed. Indeed, consider the surface $S$ mentioned in the introduction; the equation for $S\subseteq\bP^3$ is given by
\[
X^q Y - X Y^q + Z^q W - Z W^q = 0.
\]
Each $\bF_q$-plane $H=T_P X$ is the tangent plane to $S$ at some $\bF_q$-point $P$, and moreover, $H\cap S$ consists of all the $q+1$ lines in $H$ passing through $P$. Note that $\deg(S)=q+1$, and Proposition~\ref{good-hyperplane-sections} does not apply to $S$.

Let us explain why $H\cap S$ consists of $q+1$ distinct lines meeting at a single point. First, $C\colonequals H\cap S$ is a plane curve with $\deg(C)=q+1$, and $C(\bF_q)=H(\bF_q)$ since $S(\bF_q)=\bP^3(\bF_q)$. We proceed according to the following two cases:

\begin{enumerate}[label=\textit{Case} \arabic*.,
wide, labelwidth=!, labelindent=0pt]

\item\textit{$C$ has a singular point $P$ defined over $\bF_q$.}

\smallskip
In this case, each $\bF_q$-line $L\subset H$ passing through $P$ meets $C$ in at least $q+2$ points counted with multiplicity, because $\#L(F_q)=q+1$ and the intersection multiplicity of $P$ in $L\cap C$ is at least $2$. By B\'ezout's theorem, $L$ is an irreducible component of $C$. As there are $q+1$ $\bF_q$-lines in the plane passing through $P$ and $\deg(C)=q+1$, it follows that $C$ must be the union of these lines.

\medskip

\item\textit{$C$ is smooth at every point $P$ defined over $\bF_q$.}

\smallskip
Given $P\in C(\bF_q)$, the tangent line $T_P C$ meets $C$ in at least $q+2$ points counted with multiplicity, because $P$ contributes at least $2$ to the intersection. Since $\deg(C)=q+1$, B\'ezout's theorem guarantees that $T_P C \subseteq C$. If $P$ and $Q$ are distinct $\bF_q$-points of $C$, the tangent lines $T_P C$ and $T_Q C$ must coincide, or else the intersection would be a singular $\bF_q$-point of $C$, contradiction. Therefore, each of the $q^2+q+1$ points in $C(\bF_q)=\bP^2(\bF_q)$ must share the same tangent line $L$. This is impossible, as the line $L$ can only be tangent to at most $\#L(\bF_q)=q+1$ distinct $\bF_q$-points. We see that Case 2 does not happen.
\end{enumerate}
\end{Eg}

We proceed to prove Theorem~\ref{Teo:any-smooth-surface}. The argument is similar to the proof for the Frobenius classical surfaces.

\subsection{Estimate for the number of special tangent lines}

Using Lemma~\ref{prop:any-smooth-surface}, $S\cap S_{1, 0}$ or $S\cap S_{2, 0}$ is a curve. In the former case, we have already found a transverse $\bF_q$-line provided that $q\geq cd$ by Theorem~\ref{Thm:general_case}. From now on, we will assume that $S\cap S_{2, 0}$ is at most $1$-dimensional. We follow the same strategy described in Section~\ref{subsection-general-proof}. Consider again
\[
\Pi \colonequals S\cap S_{1,0}\cap S_{0, 1}\cap S_{2, 0}
\]
As before, we write $\Pi=\Pi_0\cup \Pi_1$ where $\dim\Pi_0 =0$ and $\dim\Pi_1=1$.

By Lemma~\ref{Lemma:weird_curve}, the component $\Pi_1$ entirely consists of $\bF_q$-lines. By repeating the same argument in Section~\ref{subsect:estimateSpTan_General}, we use Lemma~\ref{Lemma:spTanToTangency_General} to get the following upper bound
\[
\# \{\text{special tangents}\} \leq \frac{\#\Pi_0(\overline{\bF_q})}{2}.
\]
We will now estimate $\#\Pi_0(\overline{\bF_q})$. Consider the following scheme:
\[
A \colonequals (S\cap S_{2, 0})\setminus \Pi_1
\]
After decomposing $A$ into geometrically irreducible components,
\[
A = A_1 \cup \cdots \cup A_{m},
\]
we obtain
\[
\Pi_0 = A \cap (S_{1, 0}\cap S_{0, 1}) = 
\bigcup_{i=1}^{m} A_i\cap S_{1, 0}\cap S_{0, 1}
\]
Therefore, 
\begin{align*}
    \# \Pi_0(\overline{\bF_q}) &\leq \sum_{i=1}^{m} \# (A_i\cap S_{1, 0}\cap S_{0, 1})(\overline{\bF_q})  \\
    &\leq \sum_{i=1}^{m} \deg A_i \cdot \max\{\deg S_{1, 0}, \deg S_{0, 1}\}  \\
    &=\deg(A) \max\{\deg S_{1, 0}, \deg S_{0, 1}\} \\ 
    &\leq d(q^2+d-1)\max\{q+d-1, dq-q+1\}
\end{align*}
where in the second inequality we use the fact that $A_i\cap S_{1,0}$ or $A_i\cap S_{0,1}$ must be of dimension $0$. It is also clear that $dq-q+1\geq q+d-1$ for $q\geq d\geq 2$. This gives us the upper bound we need:
\begin{equation}\label{special-tangents-bound-arbitrary-surface}
    \begin{aligned}
\#\{ \text{special tangents} \} &\leq \frac{\#\Pi_0(\overline{\bF_q})}{2}  \\
&\leq \frac{d(q^2+d-1)(qd-q+1)}{2}
 \end{aligned}
\end{equation}

\subsection{Estimate for the number of transverse lines}

Combining (\ref{lemma:Homma-Kim}) and (\ref{special-tangents-bound-arbitrary-surface}), we obtain an upper bound on the number of tangent lines:
\begin{align*}
\#\{\text{tangent lines}\} &= \#\{\text{rational tangents}\} + \#\{\text{special tangents}\} \\
&\leq \# S(\bF_q)\cdot(q+1) + \frac{d(q^2+d-1)(qd-q+1)}{2} \\
&\leq (q+1)^2 (qd-q+1) + \frac{d(q^2+d-1)(qd-q+1)}{2} 
\end{align*}

Recall that the total number of $\bF_q$-lines in $\bP^3$ is $(q^2+1)(q^2+q+1)$,
so the number of transverse lines is bounded below as
\begin{align*}
    &\#\{\text{transverse lines}\}
    = (q^2+1)(q^2+q+1) - \#\{\text{tangent lines}\}\\
    &\geq (q^2+1)(q^2+q+1)
    - (q+1)^2 (qd-q+1)\\
    &\hspace{12pt} - \frac{d(q^2+d-1)(qd-q+1)}{2}\\
    &= q^4 - \frac{1}{2}\left[
    (d^2 + d - 4)q^3
    + (5d - 6)q^2
    + d(d^2 - 2d + 3)q
    + d(d-1)
    \right],
\end{align*}
and the existence of a transverse 
$\bF_q$-line will be deduced if the last expression is positive.

We will substitute $q=xd^2$ into the last expression, and find out the smallest permissible value of $x$ for which it is positive.
After a rearrangement:
\[
(2x^4-x^3)d^8-x^3 d^7 + 4x^3 d^6 + (-5x^2-x) d^5 + (6x^2+2x) d^4 - 3x d^3 - d^2 + d  > 0
\]
We claim that this inequality is satisfied for $x\geq 1$ and $d\geq 2$. We can group the terms:
\begin{align*}
&(2x^4-x^3)d^8-x^3 d^7 = (x^4-x^3)d^8 + (xd-1)x^3d^7 > 0 \\
& 4x^3 d^6 + (-5x^2-x) d^5 \geq 8x^3 d^5-5x^2d^5-x d^5 > 0 \\
&(6x^2+2x) d^4 - 3xd^3 - d^2 \geq 12x^2 d^3 + 8xd^2 -3xd^3-d^2 > 0
\end{align*}
Thus, $x\geq 1$ is a sufficient condition for the main inequality above to hold. We conclude that there exists an $\bF_q$-line transverse to $S$ whenever $q\geq d^2$.

\section{Frobenius classical hypersurfaces}\label{section:Frobenius-classical}

This section is devoted to proving the following implication: a smooth reflexive hypersurface is necessarily Frobenius classical. As a consequence of this result, Theorem~\ref{main-theorem} follows from  Theorem~\ref{Thm:general_case}. While the rest of the paper focuses on the case of surfaces, the results in the present section apply to any hypersurface in $\bP^n$. The definition for a hypersurface to be Frobenius (non-)classical is generalized immediately from Definition~\ref{def:frobeniusClassical}.

\begin{Def}\label{def:frobeniusNonClassical}
A projective hypersurface $X\subseteq \bP^n$ is called \textit{Frobenius non-classical} if for each smooth point $P\in X(\overline{\bF_q})$, we have $\Phi(P)\in T_{P}X$. Here $\Phi\colon \bP^n\to \bP^n$ is the usual Frobenius morphism given by 
$$
[X_0, X_1, \cdots , X_n]\mapsto [X_0^q, X_1^q, \cdots , X_n^q].
$$ Otherwise, $X$ is called \emph{Frobenius classical}.
\end{Def}

\subsection{Preliminary on the reflexivity}\label{subsect:pre-reflexive}
Supoose that $X\subseteq\bP^{n}$ is a projective variety.
Let $X_{\text{sm}}\subseteq X$ be the smooth locus.
The \textit{conormal variety} of $X$ is defined as follows:
\[
    C(X) \colonequals
    \overline{\left\{
        (P, H^{\ast}):
        P\in X_{\operatorname{sm}}
        \text{ and }
        T_{P}X\subseteq H
    \right\}}
    \subseteq X\times (\bP^{n})^{\ast}.
\]
It has two natural projections
\[\xymatrix{
    & C(X)\ar[dl]_{\pi_1}\ar[dr]^{\pi_2} &\\
    X && (\bP^n)^{\ast}.
}\]
The second projection $\pi_2$ is called the \emph{conormal map},
and its image $X^{\ast} \colonequals \pi_2(C(X))\subseteq(\bP^n)^{\ast}$ is called the \emph{dual variety} of $X$.

\begin{Def}\label{def:reflexive}
A variety $X$ is called \emph{reflexive} if $C(X)\cong C(X^{\ast})$ under the natural isomorphisms
\[
    \bP^{n} \times (\bP^{n})^{\ast}
    \cong 
    (\bP^{n})^{\ast}\times\bP^{n}
    \cong
    (\bP^{n})^{\ast}\times(\bP^{n})^{\ast\ast}.
\]
\end{Def}

The celebrated theorem of Monge-Segre-Wallace asserts that $X$ is reflexive if and only if the second projection
\[
    \pi_2\colon C(X)\to X^{\ast}
\]
is separable, i.e. the induced field extension $k(X^{\ast})\hookrightarrow k(C(X))$ is separable. The details can be found in \cite{Kleiman-tangency}. In particular, all varieties in characteristic $0$ are reflexive.

If $X$ is a hypersurface, then $\pi_1$ is birational, and the composition $\pi_2\circ\pi_1^{-1}$ coincides with the \emph{Gauss map}
\[
    \gamma\colon X\dashrightarrow (\bP^{n})^{\ast}
    : P\mapsto T_P X.
\]
As an immediate consequence of the Monge-Segre-Wallace theorem,
$X$ is reflexive if and only if the Gauss map is separable onto its image.
This applies in particular in our situation when $S$ is a surface in $\bP^3$.

\begin{Rmk}
In general, when $X$ is not a hypersurface in $\bP^n$, the Gauss map \[
    \gamma\colon X\dashrightarrow \mathbb{G}\mathrm{r}(\dim X, n)
\]
does not coincide with $\pi_2\circ\pi_1^{-1}$.
However, one implication is still true:
if a projective variety is reflexive, then the Gauss map is separable \cite{Kleiman-Piene}. It is a remarkable result of Fukasawa and Kaji \cite{Fukasawa-Kaji} that the converse holds for surfaces in $\bP^n$ for all $n\geq 3$. The converse fails in higher dimensions: Fukasawa \cite{Fukasawa-counterexample} found an example of a smooth non-reflexive projective variety whose Gauss map is separable (in fact, an embedding).
\end{Rmk}

\subsection{Examples of non-reflexive hypersurfaces}
\label{subsect:nonreflexiveExamples}
Pick $n+1$ homogeneous polynomials of degree $r$:
\[
    T_0, T_1, ..., T_n \in \bF_{q}[X_0, ..., X_n].
\]
Write $q=p^{e}$ for some prime $p$. Consider the hypersurface $X\subseteq\bP^n$ defined by the polynomial
\[
F \colonequals X_0 T_0^p +  X_1 T_1^p + \cdots + X_n T_n^p 
\]
Then $X$ is a non-reflexive hypersurface of degree $d=1+rp$. Indeed, $\frac{\partial F}{\partial X_i} = T_i^p$, and so the Gauss map $\gamma: X\dashrightarrow (\bP^n)^{\ast}$ is given by 
\[
Q\mapsto [T_0(Q)^p, ..., T_n(Q)^{p}]
\]
for each smooth point $Q\in X(\overline{\bF_q})$. The Gauss map of $X$ is inseparable, and so $X$ is non-reflexive. Note that $X$ is smooth if and only if 
\[
\{T_0=0\}\cap \cdots \cap \{T_n=0\}
\] 
is empty. Thus, we can obtain smooth non-reflexive hypersurfaces by choosing the polynomials $T_{i}$ carefully.

As an explicit example, one can choose $T_i = X_i$; the resulting variety is the \textit{Fermat hypersurface}:
\[
    X_0^{p+1} + X_1^{p+1} + \cdots + X_n^{p+1} = 0.
\]
Furthermore, this example is Frobenius non-classical over the field $\mathbb{F}_{p^2}$ since it can be checked that $\Phi^2(P)\in T_P(X)$ for each $P\in X(\overline{\bF_p})$.

\begin{Rmk}
The equations for all smooth Frobenius non-classical plane curves have been found by Hefez and Voloch \cite[Theorem~2]{Hefez-Voloch}.
\end{Rmk}

\subsection{Reflexivity implies Frobenius classicality}
Let $X\subseteq \bP^n$ be a geometrically irreducible and reduced hypersurface defined over a finite field $\bF_{q}$. Then $X$ is defined by a single homogeneous polynomial $F(X_0, X_1, \cdots , X_n)$. Following the same notation in Section~\ref{subsect:auxillarySurfaces}, the auxiliary hypersurface $X_{1,0}$ can be defined by
\begin{equation}\label{equation:of:X_10}
    G \colonequals X_0^q F_{0} + X_1^q F_{1} + \cdots + X_n^q F_{n} = 0,
\end{equation}
where $F_{i} \colonequals \partial F/\partial X_{i}$. We are interested in the following question: 
\[
    \text{When is $X$ Frobenius non-classical?}
\]
or equivalently, 
\[
    \text{When does $F$ divide $G$?}
\]
The following theorem reflects this condition.

\begin{Teo}\label{Teo:Frob:reflexive}
Let $X\subseteq \bP^n$ be a geometrically irreducible and reduced hypersurface defined over $\bF_{q}$.
If $X$ is reflexive, then it is Frobenius classical.
\end{Teo}

The analogue of Theorem~\ref{Teo:Frob:reflexive} in the case of curves is well-known to the experts \cite[Proposition~1]{Hefez-Voloch}. The proof of Theorem~\ref{Teo:Frob:reflexive} relies on Lemma~\ref{Prop:Frob-non-cl-Hessian-0} below.

\begin{proof}[Proof of Theorem~\ref{Teo:Frob:reflexive}]
If $X$ is reflexive, then the Gauss map $\gamma\colon X\dashrightarrow X^{\ast}$ is separable (in fact birational), so the ramification locus is of codimension 1. The ramification points of $\gamma$ precisely correspond to the points on $X$ where the Hessian determinant vanishes (when $X$ is a surface, such points are called \textit{parabolic points} of $X$ \cite{Voloch-surfaces}). In particular, the Hessian determinant of $X$ cannot identically vanish on all of $X$. According to Lemma~\ref{Prop:Frob-non-cl-Hessian-0} below, $X$ must then be Frobenius {classical}.
\end{proof}

\begin{Lemma} \label{Prop:Frob-non-cl-Hessian-0}
Let $X\subseteq\bP^n$ be a geometrically irreducible and reduced projective hypersurface defined by $F\in \bF_{q}[X_0, ..., X_n]$,
and let
\[
    H_F \colonequals \left(\frac{\partial^2 F}{\partial X_i \partial X_j} \right)
\]
be the Hessian matrix.
If $X$ is Frobenius non-classical, then $\det(H_F)$ vanishes identically on $X$.
\end{Lemma}

\begin{proof}
In the following, we use $F_{ij}$ as a shorthand for the partial derivative $\frac{\partial^2 F}{\partial X_i \partial X_j}$.
According to our assumption, $X$ is Frobenius non-classical, so that $X\subseteq X_{1,0}\colonequals \{G=0\}$ where $G$ is defined in (\ref{equation:of:X_10}). Since $X$ is irreducible, there is a homogeneous polynomial $R$ such that  
\[
    G = FR.
\]
 Consider the partial derivatives of $G$ with respect to each variable $X_i$:
\begin{align*}
    G_0 &= X_0^q F_{00} + X_1^q F_{01} + \cdots + X_n^q F_{0n}
    = F R_0 + F_0 R\\
    G_1 &= X_0^q F_{10} + X_1^q F_{11} + \cdots + X_n^q F_{1n}
    = F R_1 + F_1 R\\
    &\ \ \vdots
    \hphantom{X_0^q F_{n0} + X_1^q F_{n1} + \cdots + X_n^q F_{nn}}
    \ \ \vdots  \\
    G_n &= X_0^q F_{n0} + X_1^q F_{n1} + \cdots + X_n^q F_{nn}
    = F R_n + F_n R 
\end{align*}
Given $P\in X\subseteq X_{1, 0}$, let $\bm{x}=[x_0, ..., x_n]\in (\overline{\bF_q})^n$ denote the vector representing $P$. After substituting the coordinates of $P$, and using the fact that $F(\bm{x})=0$, the system above becomes
\begin{align*}
    x_0^q F_{00}(\bm{x}) + x_1^q F_{01}(\bm{x}) + \cdots + x_n^q F_{0n}(\bm{x})
    & = R(\bm{x}) F_0(\bm{x})\\
    x_0^q F_{10}(\bm{x}) + x_1^q F_{11}(\bm{x}) + \cdots + x_n^q F_{1n}(\bm{x})
    & = R(\bm{x}) F_1(\bm{x})\\
    &\ \ \vdots  \\
    x_0^q F_{n0}(\bm{x}) + x_1^q F_{n1}(\bm{x}) + \cdots + x_n^q F_{nn}(\bm{x})
    & = R(\bm{x}) F_n(\bm{x})
\end{align*}
Using a matrix notation, this is equivalent to 
\begin{equation}\label{Hessian-equation1}
H_{F}(P) [x_0^q, ..., x_n^q]^T = R(\bm{x})[F_0(\bm{x}), ..., F_n(\bm{x})]^{T}
\end{equation}
where $\bm{\nu}^T$ stands for the transpose of vector $\bm{\nu}$.

On the other hand, applying Euler's formula to the homogeneous polynomial $F_i$ for $i=0, 1, ..., n$, we obtain
\[
    X_0 F_{i0} + X_1 F_{i1} + \cdots + X_{n} F_{in} = (d-1) F_i
\]
After substituting the coordinates of $P$, this translates into the matrix equation
\begin{equation}\label{Hessian-equation2}
H_{F}(P)[x_0, ..., x_n]^T = (d-1)[F_0(\bm{x}), ..., F_n(\bm{x})]^T
\end{equation}
We discuss different situations that can arise:

\begin{enumerate}[label=\textit{Case}~\arabic*.,
wide, labelwidth=!, labelindent=0pt]

\item \textit{$F$ does not divide $R$ and $d-1\neq 0$ in $\bF_q$.}
\smallskip

In this case, we can choose $P\in X$ general enough such that $R(P)\neq0$ and $\Phi(P)\neq P$. Then the two vectors
\begin{align*}
    \bm{x}_1 &=\frac{1}{d-1}(x_0, ..., x_n)^T, \\
    \bm{x}_2 &=\frac{1}{R(\bm{x})}(x_0^q,x_1^q, ..., x_n^q)^T
\end{align*}
are linearly independent solutions to the equation $H_F(P)(\bm{x})=\bm{y}$ with 
\[
        \bm{y}=(F_0(\bm{x}), ..., F_n(\bm{x}))^T.
\]
In particular, $\det(H_F(P))=0$.

\medskip
\item \textit{$d-1= 0$ in $\bF_q$}.
\smallskip

Then for every point $P\in X$, 
\[
    \bm{x}=(x_0,x_1,...,x_n)^T
\]
is a nonzero solution to the equation $H_F(P)(\bm{x})=\bm{0}$ by (\ref{Hessian-equation2}), and so $\det(H_F(P))=0$.

\medskip
\item \textit{$F$ divides $R$}.
\smallskip
 
Then for every point $P\in X$,  
\[
    \bm{x}=(x_0^q, x_1^q, ... ,x_n^q)^T
\]
is a nonzero solution to the equation $H_F(P)(\bm{x})=\bm{0}$ by (\ref{Hessian-equation1}), and so $\det(H_F(P))=0$.
\end{enumerate}

\medskip
Combining the observations above, we deduce that if $X$ is a Frobenius non-classical hypersurface, then a general point $P\in X$ is contained in the variety defined by 
 $$
 \det(H_F)=0.
 $$
 Since $X$ is closed, and $\{\det(H_F)=0\}$ contains a nonzero open subset of $X$, it immediately follows that $X\subseteq \{\det(H_F)=0\}$.
\end{proof}

\bibliographystyle{abbrv}
\bibliography{TranLineSurf_bib}

\end{document}